\documentclass[11pt]{amsart}
\usepackage[a4paper, total={5.35 in, 8.3 in}]{geometry}
\usepackage{amsmath,amsfonts,amssymb}

\usepackage{amsfonts}

\usepackage{verbatim}
\usepackage{enumerate}
\usepackage{amsthm}
\usepackage{url}
\usepackage{multicol}
\usepackage[colorlinks]{hyperref}
\usepackage[english]{babel}
\usepackage{tikz}

\newcommand{\PGL}{\mathrm{PGL}}
\newcommand{\K}{\mathbb K}
\newcommand{\ok}{\overline{\K}}

\newcommand{\F}{\mathbb{F}}

\newcommand{\fq}{\mathbb{F}_q}

\newtheorem{theorem}{Theorem}[section]

\newtheorem{proposition}[theorem]{Proposition}
\newtheorem{definition}[theorem]{Definition}
\newtheorem{problem}{Problem}
\newtheorem{lemma}[theorem]{Lemma}
\newtheorem{corollary}[theorem]{Corollary}

\newtheorem{remark}[theorem]{Remark}

\author[J. A. Oliveira]{Jos\'e Alves Oliveira}
\address{Departamento de Matem\'{a}tica,
Universidade Federal de Minas Gerais,
UFMG,
Belo Horizonte MG (Brazil),
 30123-970}

\email{joseufmg@gmail.com}

\author[D. Oliveira]{Daniela Oliveira}
\address{Departamento de Matem\'{a}tica,
Universidade Federal de Minas Gerais,
UFMG,
Belo Horizonte MG (Brazil),
 30123-970}

\email{danielaalvesoliveira@gmail.com}

\author[L. Reis]{Lucas Reis}
\address{Departamento de Matem\'{a}tica,
Universidade Federal de Minas Gerais,
UFMG,
Belo Horizonte MG (Brazil),
 30123-970}

\email{lucasreismat@ufmg.br}

\date{\today
}

\title{On iterations of rational functions over perfect fields}

\begin{document}

\begin{abstract}
Let $\mathbb K$ be a perfect field of characteristic $p\ge 0$ and let $R\in \mathbb K(x)$ be a rational function. This paper studies the number $\Delta_{\alpha, R}(n)$ of distinct solutions of $R^{(n)}(x)=\alpha$ over the algebraic closure  $\overline{\K}$ of $\K$, where $\alpha\in \ok$ and $R^{(n)}$ is the $n$-fold composition of $R$ with itself. With the exception of some pairs $(\alpha, R)$, we prove that $\Delta_{\alpha, R}(n)=c_{\alpha, R}\cdot d^n+O_{\alpha, R}(1)$ for some $0<c_{\alpha, R}\le 1<d$. The number $d$ is readily obtained from $R$ and we provide estimates on $c_{\alpha, R}$. Moreover we prove that the exceptional pairs $(\alpha, R)$  satisfy $\Delta_{\alpha, R}(n)\le 2$ for every $n\ge 0$, and we fully describe them. We also discuss further questions and propose some problems in the case where $\K$ is finite.
\end{abstract}
\maketitle

\section{Introduction}
For a field $\K$ and a rational function $R\in \K(x)$, we set $R^{(0)}(x)=x$ and, for $n\ge 1$, $R^{(n)}(x)=R^{(n-1)}(R(x))$. The rational function $R^{(n)}(x)\in \K(x)$ is the $n$-th iterate of $R$. When $R=f$ is a polynomial, the compositions $f^{(n)}(x)$ are also polynomials. The iterates of polynomials have been extensively studied in the past few years~\cite{A05,AM00,J08,JB12,K85}; in many of the cases, the authors explore the {\em stable polynomials}. These are the polynomials $f\in \K[x]$ in which all the iterates $f^{(n)}(x), n\ge 1$ are irreducible over $\K$. When $\K$ is finite, the concept of stability is naturally extended to a set $\{f_1, \ldots, f_r\}$ of polynomials~\cite{HBM}. Still in the finite field case, further arithmetic properties of the polynomial iterates $f^{(n)}$ are studied in~\cite{GOS}. The authors explore the number of distinct roots, the number of  irreducible factors over $\K$ and the largest degree of an irreducible factor of $f^{(n)}$ over $\K$. In particular they prove that, under some mild conditions on $f$, those three functions grow (roughly) at least linearly with respect to $n$. 

Some results of~\cite{GOS} were recently improved and extended to iterates $f(g^{(n)}(x))$ in~\cite{R}. Most notably, in~\cite{R} it is proved that up to some exceptional pairs $(f, g)$,  the number $\Delta_n$ of distinct roots of $f(g^{(n)}(x))$ actually grows exponentially. More precisely, the inequality $c_1d^n\le \Delta_n\le c_2d^n$ holds for every sufficiently large $n$, where $c_1, c_2>0$ and $d>1$ do not depend on $n$. However, only the constant $d$ is explicitly given there, making the estimate imprecise. The exceptional pairs $(f, g)$ are fully described and it is direct to verify that, for such pairs, the numbers $\{\Delta_n\}_{ n\ge 0}$ are uniformly bounded by a constant. For more details, see Section~2 of~\cite{R}. Many other arithmetic aspects of the iterates $f(g^{(n)}(x))$ are also studied in~\cite{R}, mainly motivated by Question 18.9 in \cite{B}; this question includes a more general setting, allowing $g$ to be a rational function.

In the context of rational functions, the iterates $f(R^{(n)}(x))$  have not been much explored, but we can naturally extend questions and definitions from the polynomial setting. For instance, if $R_n:=R^{(n)}(x)=g_n/h_n$ with $g_n, h_n$ relatively prime polynomials, we define the polynomial $f_{R}^{(n)}=h_n^{\deg(f)}f(R_n)$. So we may consider the notion of $R$-stability, meaning that $f$ is $R$-stable if all the polynomials $f_R^{(n)}(x)$ are irreducible for every $n\ge 0$. The $R$-stability of polynomials was recently explored for a special class of rational functions $R$ when $\K$ is finite~\cite{PRW}. 

The aim of this paper is to refine the main result in~\cite{R}, extending it to a more general setting. We consider $\K$ a perfect field, $R\in \K(x)$ a rational function of positive degree and study the number $\Delta_{\alpha, R}(n)$ of distinct solutions of $R^{(n)}(x)=\alpha$ over the algebraic closure $\overline{\K}$ of $\K$, where $\alpha\in \ok$.
Our main results, Theorems~\ref{thm:main} and~\ref{thm:sec}, not only recovers the exponential bound in~\cite{R} but also provides a more precise estimate on $\Delta_{\alpha, R}(n)$. We prove that, with the exception of some pairs $(\alpha, R)$, the equality $\Delta_{\alpha, R}(n)=c_{\alpha, R}\cdot d^n+O_{\alpha, R}(1)$ holds for some $0<c_{\alpha, R}\le 1<d$. The parameter $d$ is easily obtained from $R$ and there is an implicit formula for $c_{\alpha, R}$; in particular, we provide estimates on $c_{\alpha, R}$ by means of simple parameters. Similarly to the polynomial case~\cite{R}, the exceptional pairs $(\alpha, R)$ satisfy $\Delta_{\alpha, R}(n)\le 2$ for every $n\ge 0$, and are fully described. However, in contrast to the polynomial setting, we have many more pathological situations; for more details, see Theorem~\ref{thm:sec}. We also discuss the growth of some arithmetic functions related to the factorization of $f_{R}^{(n)}(x)$ when $\K$ is finite, extending some minor results and open problems from~\cite{R}. 

The main idea behind the proof of Theorems~\ref{thm:main} and~\ref{thm:sec} is to provide an implicit formula for $\Delta_{\alpha, R}(n)$, considering the number $r_{\beta, R}$ of solutions of $R(x)=\beta$ with $\beta$ ranging over the elements in $\ok$ such that $R^{(i)}(\beta)=\alpha$ for some $i\ge 0$. With the exclusion of some exceptional $R$'s, we prove that $\Delta_{\alpha, R}(n)=c_{\alpha, R}d^n+O_{\alpha, R}(1)$ for some $0\le c_{\alpha, R}\le 1<d$, where $c_{\alpha, R}$ depends on the numbers $r_{\beta, R}$. We then estimate $c_{\alpha, R}$ by means of parameters such as the degree of the extension $\K(\alpha)/\K$ and the degree of the Wronskian associated to $R$. This allows us to describe the pairs $(\alpha, R)$ in which $c_{\alpha, R}$ vanishes. Along with the exceptional $R$'s, the latter fully describes the pathological cases. 

The paper is organized as follows. In Section 2 we state our main results and provide some important remarks. Section 3 provides background material and important preliminary results. In Section 4 we prove our main results. Finally, in Section 5 we extend some open problems and minor results from~\cite{R}.

\section{Main results}
In this section we state our main results. Before doing so, we need to introduce some basic definitions.  Throughout this paper, $\K$ denotes a perfect field of characteristic $p\ge 0$ and $\overline{\K}$ denotes its algebraic closure.
By a rational function $R\in \K(x)$ we mean a quotient $\frac{g}{h}$, where $g, h\in \K[x]$ are relatively prime polynomials. For simplicity, we sometimes assume that $h$ is monic. The degree of $R$ is $\max\{\deg(g), \deg(h)\}$. Since $\K$ is perfect, if $p>0$, the Frobenius map $a\mapsto a^p$ is an automorphism of $\K$. We have the following definition.

\begin{definition}
Let $\K$ be a perfect field of characteristic $p\ge 0$ and let $R=g/h\in \K(x)$ be a rational function of degree $D\ge 1$. If $p>0$, the $p$-reduction of $R$ is the unique rational function $\tilde{R}\in \K(x)$ such that $R=\tilde{R}^{p^h}$, $h\ge 0$ and $\tilde{R}$ is not of the form $R_0^p$ with $R_0\in \K(x)$. For convention, if $p=0$,  the $p$-reduction of $R$ equals $R$ itself. For each $\alpha\in \ok$, we set $R^{-\infty}(\alpha)=\cup_{n\ge 0}\{\beta\in \ok\,|\, R^{(n)}(\beta)=\alpha\}$, the reversed $R$-orbit of $\alpha$. Also, $\alpha\in \ok$ is $R$-critical if
$$\sup_{n\ge0 }\Delta_{\alpha, R}(n)<+\infty,$$
where $\Delta_{\alpha, R}(n)$ denotes the number of distinct solutions of $R^{(n)}(x)=\alpha$ over $\ok$. 
\end{definition}
Our main results can be stated as follows.

\begin{theorem}\label{thm:main}
Let $\K$ be a perfect field of characteristic $p\ge 0$ and let $R=G/H\in \K(x)$ be a rational function whose $p$-reduction $\tilde{R}=g/h$ has degree $d>1$. Let $d'\ge 0$ be the degree of $W=g'h-gh'$, where $f'$ denotes the formal derivative of $f$. Suppose that $\alpha\in \ok$ is not $R$-critical and set $e=[\K(\alpha): \K]$.  Then there exists $0<c_{\alpha, R}\le 1$ such that 
$$\Delta_{\alpha, R}(n)=c_{\alpha, R}d^n+O_{\alpha, R}(1).$$
The constant $c_{\alpha, R}$ can be implicitly computed from the set $R^{-\infty}(\alpha)$ and we have the following estimates:

\begin{enumerate}
\item If $\alpha$ is not $R$-periodic, then $c_{\alpha, R}\ge \frac{1}{d^2}-\frac{1}{d^3}$. Moreover, 
\begin{enumerate}[(a)]
\item $c_{\alpha, R}\ge 1-\frac{d'}{de}\ge \frac{1}{d}$ if $e>1$;
\item $c_{\alpha, R}\ge 1-\frac{\min\{d-1, d'\}}{d}-\frac{d'-\min\{d-1, d'\}}{d^2}\ge \frac{1}{d^2}$ if $e=1$ and $R^{-\infty}(\alpha)$ does not contain an element $\gamma\in \K$ with $\deg(G-\gamma H)<\deg(R)$.
\end{enumerate}
\item If $\alpha$ is $R$-periodic of period $N$, then $c_{\alpha, R}\ge \frac{1}{4d^2}$. Moreover, 
\begin{enumerate}[(a)]
\item $c_{\alpha, R}\ge 1-\frac{d'}{e(d-1)}\ge \frac{1}{3}$ if $e>2$;
\item $c_{\alpha, R}\ge \frac{1}{d^2}- \frac{1}{d^3}$ if  $e=2$.
\end{enumerate}
\end{enumerate}
\end{theorem}

\begin{theorem}\label{thm:sec}
Let $\K$ be a perfect field of characteristic $p\ge 0$ and let $R=g/h\in \K(x)$ be a rational function of degree $D$ whose $p$-reduction has degree $d\ge 1$. Fix $\alpha\in \ok$ and set $e=[\K(\alpha): \K]$. Then $\alpha$ is $R$-critical if and only if one of the following holds:

\begin{enumerate}
\item $d=1$, that is, $R(x)=\frac{ax^D+b}{cx^D+d}$ with $ad-bc\ne 0$ and $D=1$ if $p=0$ or $D=p^h, h\ge 0$, otherwise.
\item $d>1, \alpha\in \K$ is not $R$-periodic and 
\begin{enumerate}[(a)]
\item $R(x)=\alpha+\frac{\lambda}{h(x)}$ for some $\lambda\in \K^*$ and some $h\in \K[x]$ of degree $D$;
\item $R(x)=\beta+\frac{\lambda}{(x-\beta)^D-\frac{\lambda}{\beta-\alpha}}$ for some $\beta, \lambda\in \K$ with $\beta\ne \alpha$ and $\lambda\ne 0$.
\end{enumerate}
\item $d>1$, $\alpha\in \ok$ is $R$-periodic of period $N$ and
\begin{enumerate}[(a)]
\item $e=N=2$ and $R(x)=\frac{\overline{\alpha}(x-\alpha)^D-\alpha(x-\overline{\alpha})^D}{(x-\alpha)^D-(x-\overline{\alpha})^D}$, where $\overline{\alpha}\ne \alpha$ is the conjugate root of the minimal polynomial of $\alpha$ over $\K$.
\item $e=2, N=1$ and $R(x)=\frac{\alpha(x-\alpha)^D-\overline{\alpha}(x-\overline{\alpha})^D}{(x-\alpha)^D-(x-\overline{\alpha})^D}$, where $\overline{\alpha}\ne \alpha$ is the conjugate root of the minimal polynomial of $\alpha$ over $\K$.

\item $e=1, d=2, N=3$ and $R(x)=\frac{y_1(x-y_1)^D-(y_1+y_2)(x-y_2)^D}{(x-y_1)^D-2(x-y_2)^D}$,
where $y_1\ne y_2$ are elements of $\K$ and $\alpha\in \{y_1, y_2, \frac{y_1+y_2}{2}\}$.

\item $e=1, N=2$ and $R(x)=\frac{\alpha(x-\alpha)^A-\beta \lambda (x-\beta)^B}{(x-\alpha)^A-\lambda(x-\beta)^B}$, where $\beta\in \K\setminus\{\alpha\}, \lambda\in \K^*$ and $A, B$ are positive integers with $\max\{A, B\}=D$.

\item $e=1, N=d=2\ne p$ and $R(x)=\beta+\frac{(\alpha-\beta)^{D+1}}{(2x-\alpha-\beta)^D+(\alpha-\beta)^D}$ for some $\beta\in \K\setminus\{\alpha\}$.

\item $e=1$, $N=d=2\ne p$ and $R(x)=\beta+\frac{2(\alpha-\beta)^{D+1}}{(x-\alpha)^D+(\alpha-\beta)^D}$ for some $\beta\in \K\setminus\{\alpha\}$.

\item $e=N=1$ and $R(x)=\alpha+\frac{(x-\alpha)^A}{h(x)}$ for some $h\in \K[x]$ and some integer $A\ge 1$ with $h(\alpha)\ne 0$ and $\max\{A, \deg(h)\}=D$.

\item $e=N=1$ and $R(x)=\frac{\beta(x-\beta)^A(x-\alpha)^{D-A}-\alpha\lambda}{(x-\beta)^A(x-\alpha)^{D-A}-\lambda}$, where $\beta\in \K \setminus\{\alpha\}, \lambda\in \K^*$ and  $1\le A< D$.

\item $e=N=1, d=2\ne p$ and $R(x)=\frac{\alpha+\beta}{2}+\frac{(\alpha-\beta)^{D+1}}{4(2x-\alpha-\beta)^D-2(\alpha-\beta)^D}$ for some $\beta \in \K \setminus\{\alpha\}$.
\end{enumerate}
\end{enumerate}
In particular, if $\alpha$ is $R$-critical, the inequality $\Delta_{\alpha, R}(n)\le 2$ holds for every $n\ge 0$ and the reversed $R$-orbit of $\alpha$, $R^{-\infty}(\alpha)$, is finite if and only if one of the following holds:
\begin{enumerate}
\item  $d\ne 1$; 
\item $d=1$ and $\alpha$ is $R$-periodic; 
\item $d=1$ and $R(x)=\frac{ax^D+b}{cx^D+d}$ with $c\ne 0$, and $\frac{a}{c}\in R^{-\infty}(\alpha)$.  
\end{enumerate}
\end{theorem}

Theorems~\ref{thm:main} and~\ref{thm:sec} entail that the arithmetic function $\Delta_{\alpha, R}(n)$ is either uniformly bounded by a constant or grows exponentially. The following corollary is a straightforward application of Theorems~\ref{thm:main} and~\ref{thm:sec}  to the case where $R$ is a polynomial.

\begin{corollary}
Let $\K$ be a perfect field of characteristic $p\ge 0$, $\alpha\in \ok$ with $[\K(\alpha): \K]=e$ and let $f\in \K[x]$ be a $D$-degree polynomial whose $p$-reduction $F$ has degree $d>1$. Furthermore, assume that $f$ is not of the form $a(x-\alpha)^D+\alpha$ for some $a\in \K$ and set $d'=\deg(F')\le d-1$. Then there exists a constant $0< c_{\alpha, f}\le 1$ such that 
$$\Delta_{\alpha, f}(n)=c_{\alpha, f}d^n+O_{\alpha, f}(1).$$
Moreover,  $c_{\alpha, f}\ge \frac{1}{4d^2}$ if $\alpha$ is $f$-periodic and $c_{\alpha, f}\ge 1-\frac{d'}{d}\ge \frac{1}{d}$, otherwise.
\end{corollary}

\section{Preparation}
In this section we provide some definitions and important preliminary results. Throughout this section, unless otherwise stated, $R\in \K(x)$ stands for a rational function of degree $D$ whose $p$-reduction has degree $d\ge 1$.

\begin{definition}
Let $R=f/g\in \K(x)$ be a rational function of degree $D\ge 1$ and $\alpha\in \ok$.
\begin{enumerate}[(i)]
\item $r_{\alpha, R}\ge 0$ is the number of distinct roots of $g-\alpha h$ over $\ok$;
\item $\alpha$ is $R$-trivial if the polynomial $g-\alpha h$ has degree at most $D-1$.
\item $\alpha$ is $R$-periodic if there exists an integer $N\ge 1$ such that $R^{(N)}(\alpha)=\alpha$. In affirmative case, the smallest integer with this property is the period of $\alpha$.
\end{enumerate}
\end{definition}

\begin{definition}
For a rational function $R\in \K(x)$ of degree $D$ whose $p$-reduction has degree $d\ge 1$, let $\sigma_R$ be the unique automorphism of $\ok$ satisfying $\sigma_R(a^{D/d})=a$ for every $a\in \ok$. 
\end{definition}

\begin{remark}\label{rem:aut}We observe that $\sigma_R$ is the identity map if $d=D$. If $d\ne D$, then $\K$ has characteristic $p>0$ and $\sigma_R$ is just the inverse of a power of the Frobenius automorphism $a\mapsto a^p$. Furthermore, for $y, \alpha\in \ok$, we have that $R(y)=\alpha$ if and only if $\tilde{R}(y)=\sigma_R(\alpha)$, where $\tilde{R}$ is the $p$-reduction of $R$. 
\end{remark}

The following result is straightforward.

\begin{lemma}\label{lem:basic}
Let $R=f/g\in \K(x)$ be a rational function and let $\tilde{R}$ be its $p$-reduction, $d=\deg(\tilde{R})$. Then for every $\alpha\in \ok$, we have that $r_{\alpha, R}=r_{\sigma_R(\alpha), \tilde{R}}$. In particular, $r_{\alpha, R}\le d$ for every $\alpha\in \ok$.
\end{lemma}

\begin{definition}
Let $R\in \K(x)$ and $\alpha\in \ok$. For each $n\ge 0$, set $R^{[-n]}(\alpha)=\{\beta\in \ok\,|\, R^{(n)}(\beta)=\alpha\}$ and let $R^{[-n]}(\alpha)^{*}$ be the set of elements $\beta\in R^{[-n]}(\alpha)$ such that no element $R^{(i)}(\beta)$ with $0\le i\le n-1$ is $R$-periodic. 
Moreover, we set $\Delta_{\alpha, R}(n)=\# R^{[-n]}(\alpha)$ and $\Delta_{\alpha, R}(n)^{*}=\# R^{[-n]}(\alpha)^{*}$.
\end{definition}

In the proof of our main results, an implicit formula for $\Delta_{\alpha, R}(m)^*$ is required. In this context, the following definition is crucial.

\begin{definition}\label{def:main}
Let $R\in \K(x)$ be a rational function whose $p$-reduction has degree $d>1$. For each $\alpha \in \ok$ and each integer $j\ge 2$, set $$n_{\alpha, j}(R)=\sum_{\gamma\in R^{[1-j]}(\alpha)^{*}} (d-r_{\gamma, R})\ge 0.$$ For convention, we set $n_{\alpha, 1}(R)=d-r_{\alpha, R}+1$ if $\alpha$ is $R$-periodic and $n_{\alpha, 1}(R)=d-r_{\alpha, R}$, otherwise.
\end{definition}

We obtain the following result.

\begin{proposition}\label{prop:aux}
Let $R\in \K(x)$ be a rational function whose $p$-reduction has degree $d>1$. Then for every $m\ge 1$ and every $\alpha\in \ok$, we have that 
$$\Delta_{\alpha, R}(m)^*=d^{m}-\sum_{j=1}^{m}n_{\alpha, j}(R)\cdot d^{m-j}=d^m\left(1-\sum_{j=1}^{m}n_{\alpha, j}(R)d^{-j}\right).$$
\end{proposition}

\begin{proof}
	We proceed by induction on $m$. The case $m=1$ follows directly by the definition of $n_{\alpha, 1}(R)$. Suppose that the result holds for an integer $m\ge 1$. We observe that the elements of $R^{[-m-1]}(\alpha)^*$ comprise the roots of $R(x)=\gamma$ with $\gamma\in R^{[-m]}(\alpha)$. The latter implies that
	$$\Delta_{\alpha, R}(m+1)^*=d \Delta_{\alpha, R}(m)^*-\sum_{\gamma\in R^{[-m]}(\alpha)^{*}} (d-r_{\gamma, R})= d\Delta_{\alpha, R}(m)^*-n_{\alpha, m+1}(R),$$
from where the result follows.
\end{proof}

In the following proposition we provide estimates on the numbers $n_{\alpha, j}(R)$.

\begin{proposition}\label{lemma:1}
Let $R\in \K(x)$ be a rational function  whose $p$-reduction $\tilde{R}=g/h$ has degree $d>1$. For each $\alpha\in \ok$ 
set $\kappa_{\alpha, R}=\sum_{j\ge 1}n_{\alpha, j}(R)$,
and let $\delta_{\alpha, R}=1$ or $0$, according to whether $\alpha$ is $R$-periodic or not, respectively. If $d'=\deg(g'h-gh')$, the following hold:
 
 \begin{enumerate}[(i)]
 \item for distinct elements $\alpha_1, \ldots, \alpha_{\ell}\in \ok$, we have that $$\sum_{i=1}^{\ell}\kappa_{\alpha_i, R}\le \varepsilon+\sum_{i=1}^{\ell}\delta_{\alpha_i, R},$$ 
where $\varepsilon=d'$ if no set $R^{-\infty}(\alpha_i)$ contains an $R$-trivial element and $\varepsilon=2d-1$, otherwise;

 \item $\kappa_{\alpha, R}\le \frac{d'}{e}+\delta_{\alpha, R}$ if $[\K(\alpha): \K]=e>1$.
 
 \end{enumerate}
 \end{proposition}

\begin{proof}
From Lemma~\ref{lem:basic}, it follows that $r_{\gamma, R}\le d$. For each $\gamma\in \ok$, let $T_{\gamma}$ be the degree of $g-\sigma_R(\gamma) h$. We observe that the inequality $T_{\gamma}<d$ holds for at most one element $\gamma\in \ok$ and, in this case, we necessarily have that $\gamma\in \K$.

Since $\tilde{R}$ has degree $d$, Remark~\ref{rem:aut} entails that $d-r_{\gamma, R}>0$ if and only if $g-\sigma_R(\gamma) h$ has $(T_{\gamma}-r_{\gamma, R})$ common roots with the polynomial $g'-\sigma_R(\gamma)h'$, multiplicities counted. In particular, $g-\sigma_R(\gamma) h$ has $(T_{\gamma}-r_{\gamma, R})$ common roots with the Wronskian $W=g'h-gh'$, multiplicities counted. From construction, the polynomials $g$ and $h$ are relatively prime and their formal derivatives cannot vanish simultaneously. In particular, $W$ does not vanish and a detailed account on the possible degrees of $g$ and $h$ entails that $d'=\deg(W)\le 2d-2$.  

We prove items (i) and (ii) separately.

\begin{enumerate}[(i)]
\item We observe that the sets $R^{[1-j]}(\alpha_i)^*$ with $j\ge 1$ and $1\le i\le \ell$ are pairwise disjoint. Therefore, from the previous remarks we obtain that
\begin{equation}\label{eq:t}\sum_{i=1}^{\ell}\sum_{j\ge 1}\sum_{\gamma\in R^{[1-j]}(\alpha_i)^*}(T_{\gamma}-r_{\gamma,R})\le d'.\end{equation}
If no set $R^{-\infty}(\alpha_i)$ contains an $R$-trivial element, it follows that $T_{\gamma}=d$ for every $\gamma\in R^{-\infty}(\alpha_i)$. In this case, Eq.~\eqref{eq:t} implies that 
$$\sum_{i=1}^{\ell}(\kappa_{\alpha_i, R}-\delta_{\alpha_i, R})=\sum_{i=1}^{\ell}\sum_{j\ge 1}\sum_{\gamma\in R^{[1-j]}(\alpha_i)^*}(d-r_{\gamma,R})\le d'.$$
Suppose that $R^{-\infty}(\alpha_i)$ contains an $R$-trivial element $\lambda \in\K$ for some $1\le i\le \ell$. We have that $g(x)=\lambda h(x)+h_0(x)$, where $\deg(h_0)=s$ with $0\le s<d$. Therefore, $T_{\lambda}=s$ and a simple calculation yields $d'=\deg(W)\le d+s-1$. Since there exists at most one $R$-trivial element, we have that 
\begin{align*}\sum_{i=1}^{\ell}\left(\kappa_{\alpha_i, R}-\delta_{\alpha, i, R}\right)& =(d-s)+\sum_{i=1}^{\ell}\sum_{j\ge 1}\sum_{\gamma\in R^{[1-j]}(\alpha_i)^*}(T_{\gamma}-r_{\gamma,R}) \\ {} & \le d-s+d'\le 2d-1.\end{align*}

\item  Fix $\alpha\in \ok$ with $[\K(\alpha): \K]=e>1$, hence $\alpha\not\in \K$. Let $F$ be the minimal polynomial of $\alpha$ over $\K$ and let $\mathbb L\subseteq \ok$ be the splitting field of $F$. Since $\K$ is a perfect field, the roots $\alpha:=\alpha_1, \ldots, \alpha_e\in \ok$ of $F$ are all distinct and the extension $\mathbb L/\mathbb K$ is Galois. Since the Galois group of an irreducible polynomial acts transitively on its roots, for each $1\le i\le e$ there exists a $\K$-automorphism $\tau_i: \to L$ such that $\tau_i(\alpha)=\alpha_i$. Since $R\in \K(x)$, by extending these automorphisms to $\ok$  we conclude that $\kappa_{\alpha_i, R}=\kappa_{\alpha, R}$ and $\delta_{\alpha_i, R}=\delta_{\alpha, R}$ for every $1\le i\le e$. Since $e>1$, no element $\alpha_i$ lies in $\K$. Therefore, the sets $R^{-\infty}(\alpha_i)$ do not contain $R$-trivial elements. Applying item (i) for the elements $\alpha_1, \ldots, \alpha_e$, we obtain that
$$e\cdot \kappa_{\alpha, R}=\sum_{i=1}^e\kappa_{\alpha_i, R}\le d'+\sum_{i=1}^e\delta_{\alpha_i, R}=d'+e\cdot \delta_{\alpha, R},$$
from where the result follows.
\end{enumerate}
\end{proof}
\section{Proof of the main results}
Before proceeding to the proof of Theorems~\ref{thm:main} and~\ref{thm:sec}, we introduce a useful definition.

\begin{definition}\label{def:c}
Let $R\in \K(x)$ be a rational function of degree $D$ whose $p$-reduction has degree $d>1$. For each $\alpha\in \ok$ and $j\ge 1$, let $n_{\alpha, j}(R)$ be as in Definition~\ref{def:main}. If $\alpha$ is not $R$-periodic, we set  
$$c_{\alpha, R}=1-\sum_{j\ge 1}n_{\alpha, j}(R)d^{-j}.$$
 If $\alpha$ is $R$-periodic and $\alpha_1, \ldots, \alpha_N=\alpha$ are the distinct $R$-periodic elements in the $R$-orbit of $\alpha$, we set
$$c_{\alpha, R}=\frac{1}{d^N-1}\sum_{i=1}^Nd^i\left(1-\sum_{j\ge 1}n_{\alpha_i, j}(R)d^{-j}\right).$$
\end{definition}
Proposition~\ref{lemma:1} entails that the sum $\sum_{j\ge 1}n_{\alpha, j}(R)d^{-j}$ contains only finitely many nonzero terms; this fact is frequently used. We obtain the following estimate.

\begin{proposition}\label{prop:est}
Let $R\in \K(x)$ be a rational function whose $p$-reduction $Q$ has degree $d>1$. For every $\alpha\in \K$, we have that
$$\Delta_{\alpha, R}(n)=c_{\alpha, R}\cdot d^n+L_{\alpha, R}(n),$$
where $L_{\alpha, R}(n)=O_{\alpha, R}(1)$ and, in fact, $L_{\alpha, R}(n)=0$ if $\alpha$ is not periodic and $n$ is sufficiently large.

\end{proposition}

\begin{proof}
If $\alpha$ is not $R$-periodic we observe that, for every $n\ge 1$, we have that $\Delta_{\alpha, R}(n)=\Delta_{\alpha, R}(n)^*$. Proposition~\ref{prop:aux} implies that the equality 
$$\Delta_{\alpha, R}(n)=c_{\alpha, R}\cdot d^n,$$
holds for sufficiently large $n$. Suppose that $\alpha$ is $R$-periodic and let $\alpha_1, \ldots, \alpha_N=\alpha$ be the distinct $R$-periodic elements in the $R$-orbit of $\alpha$. By stratifying the elements $\beta\in R^{[-n]}(\alpha)$ according to how many integers $1\le i\le n$ satisfy $R^{(i)}(\beta)=\alpha$, we obtain that
$$R^{[-n]}(\alpha)=\{\alpha_{u}\}\cup \bigcup_{i=1}^N\bigcup_{1\le m \le n\atop  N|n+i-m}R^{[-m]}(\alpha_i)^*,$$
where $1\le u\le N$ and $u\equiv -n\pmod N$. It follows by the definition that the sets $R^{[-m]}(\alpha_i)^*$ are pairwise distinct and none of them contains $\alpha_u$, hence 
$$\Delta_{\alpha, R}(n)=1+\sum_{i=1}^{N}\sum_{1\le m\le n\atop N|n+i-m}\Delta_{\alpha_i, R}(m)^*.$$
Let $M$ be sufficiently large such that $\sum_{j\ge 1}n_{\alpha_i, j}(R)d^{-j}=\sum_{j=1}^Mn_{\alpha_i, j}(R)d^{-j}$ for every $1\le i\le N$. Fix an integer $t> M$,  let $n>t$ be sufficiently large with $n\equiv t\pmod N$ and set $q=\frac{n-t}{N}$. Therefore, for a constant $C=C_t$, we have that
$$\Delta_{\alpha, R}(n)-C=\sum_{i=1}^N\sum_{t<m\le n \atop N|n+i-m}\Delta_{\alpha_i, R}(m)^*=\sum_{i=1}^N\sum_{s=1}^q\Delta_{\alpha_i, R}(n-Ns+i)^*.$$
Since $n-Ns+i\ge t>M$ for every $1\le s\le q$ and every $1\le i\le N$, Proposition~\ref{prop:aux} entails that $\Delta_{\alpha_i, R}(n-Ns+i)^*=d^{n-Ns+i}\cdot \theta_i$ with $$\theta_i=1-\sum_{j\ge 1}n_{\alpha_i, j}(R)d^{-j}.$$ We conclude that
$$\Delta_{\alpha, R}(n)-C=\sum_{i=1}^N\sum_{s=1}^{q}d^{n-Ns+i}\theta_i=d^n\cdot \ell \cdot \sum_{i=1}^{N}d^i\theta_i,$$
where $\ell=\sum_{s=1}^qd^{-Ns}=\frac{1}{d^N-1}+O_{N}(d^{-n})$. By the definition, $c_{\alpha, R}=\frac{1}{d^N-1}\sum_{i=1}^{N}d^i\theta_i$, so that $\Delta_{\alpha, R}(n)=c_{\alpha, R}\cdot d^n+C_{\alpha, R, t}$. By taking $t=M+i$ with $1\le i\le N$, the error $C_{\alpha, R, t}$ is uniformly bounded by a constant $C_{\alpha, R}$.
\end{proof}
 Here we summarize the next steps in the proof of our main results. Proposition~\ref{prop:est} implies that, for $d>1$, $\alpha\in \ok$ is $R$-critical if and only if $c_{\alpha, R}=0$. By employing the bounds from Proposition~\ref{lemma:1}, we estimate the constant $c_{\alpha, R}$ and detect the possible distributions of the numbers $\{n_{\beta, j}\}_{ \beta\in R^{-\infty}(\alpha)}$ in which $c_{\alpha, R}=0$. We then characterize the pairs $(\alpha, R)$ that yield one of these distributions. Along with the generic critical case where $d=1$, the latter fully describes the $R$-critical elements.

\subsection{Proof of Theorem~\ref{thm:main}}
We consider the cases where $\alpha$ is $R$-periodic or not $R$-periodic separately.  
\subsubsection{The case where $\alpha$ is not $R$-periodic}
Recall that $c_{\alpha, R}=1-\sum_{j\ge 1}n_{\alpha, j}(R)d^{-j}$.
If $[\K(\alpha): \K]=e>1$, Proposition~\ref{lemma:1} entails that
$$c_{\alpha, R}=1-\sum_{j\ge 1}n_{\alpha, j}(R)d^{-j}\geq 1-d^{-1}\sum_{j\ge 1}n_{\alpha, j}(R)\geq 1-\frac{d'}{de}\ge \frac{1}{d},$$
where in the last inequality we used the fact that $e\ge 2$ and $d'\le 2d-2$. Suppose that $e=1$, that is, $\alpha\in \K$. We observe that $0\leq n_{\alpha, 1}(R)\leq d$ and Proposition~\ref{lemma:1} implies that $\sum_{j\ge 1}n_{\alpha, j}(R)\le 2d-1$. We obtain the following trivial configurations:
\begin{itemize}
	\item  $n_{\alpha, 1}(R)=d$;
	\item  $n_{\alpha, 1}(R)=d-1$ and $n_{\alpha, 2}(R)=d$.
\end{itemize}
In both cases, it follows that $c_{\alpha, R}=0$ and then $\alpha$ is $R$-critical. Suppose that $\alpha$ does not satisfy none of the cases described above. If $R^{-\infty}(\alpha)$ contains an $R$-trivial element, it follows that
$$c_{\alpha, R}=1-\sum_{j\ge 1}n_{\alpha, j}(R)d^{-j}\geq 1-(d-1)\cdot d^{-1}-(d-1)\cdot d^{-2}-1\cdot d^{-3}=\frac{1}{d^2}-\frac{1}{d^3}.$$
Otherwise, Proposition~\ref{lemma:1} entails that $\sum_{j\ge 1}n_{\alpha, j}(R)\le d'\le 2d-2$ and so 
$$c_{\alpha, R}=1- \sum_{j\ge 1}n_{\alpha, j}(R)d^{-j}\geq 1-\min\{d-1, d'\}\cdot d^{-1}-(d'-\min\{d-1, d'\})\cdot d^{-2}\ge \frac{1}{d^2}.$$
We combine all the previous bounds and obtain that $c_{\alpha, R}\ge \frac{1}{d^2}-\frac{1}{d^3}$ if $\alpha$ is neither $R$-periodic nor $R$-critical. This proves Theorem~\ref{thm:main} for the non periodic case.

\subsubsection{The case where $\alpha$ is $R$-periodic}
Let $\alpha_1, \ldots, \alpha_N=\alpha$ be the distinct $R$-periodic elements in the $R$-orbit of $\alpha$ and, for each $1\le i\le N$, set $\theta_i=1-\sum_{j\ge 1}n_{\alpha_i, j}(R)d^{-j}$. It follows by the definition that
$c_{\alpha, R}=\frac{1}{d^N-1}\sum_{i=1}^Nd^i\theta_i$. Moreover, $R^{-\infty}(\alpha_i)=R^{-\infty}(\alpha)$ for every $1\le i\le N$. Proposition~\ref{prop:aux} entails that each $\theta_i$ is nonnegative, hence $c_{\alpha, R}>0$ unless all the elements $\theta_i$ vanish. Proposition~\ref{lemma:1} provides the bound
\begin{equation}\label{eq:1}
\sum_{i=1}^{N}\sum_{j\ge 1}n_{\alpha_i, j}(R)\le \varepsilon+N,
\end{equation}
where $\varepsilon=d'$ if no set $R^{-\infty}(\alpha)$ contains an $R$-trivial element and $\varepsilon=2d-1$, otherwise. Set $e=[\K(\alpha): \K]$. We split the proof into cases.

\begin{enumerate}[(i)]

\item Suppose that $e>1$. It is direct to verify that $\K(\alpha_i)=\K(\alpha)$ for every $1\le i\le N$ and so $[\K(\alpha_i): \K]=e$. In particular, no set $R^{-\infty}(\alpha_i)$ contains an $R$-trivial element. Since each $\alpha_i$ is $R$-periodic, $n_{\alpha_i, 1}(R)\ge 1$ for every $1\le i\le N$. In particular, Proposition~\ref{lemma:1} implies that $1\le \sum_{j\ge 1}n_{\alpha_i, j}(R)\le d'/e+1$. For $e>2$, it follows that 
\begin{align*}c_{\alpha, R} & \ge \frac{1}{d^N-1}\sum_{i=1}^Nd^i\left(1-d^{-1}\sum_{j\ge 1}n_{\alpha_i, j}(R)\right)\\ &   \ge \frac{1}{d^N-1}\sum_{i=1}^Nd^i\left(1-\frac{d'}{de}-\frac{1}{d}\right) =1-\frac{d'}{(d-1)e}\ge \frac{1}{3},\end{align*}
since $e>2$ and $d'\le 2(d-1)$. If $e=2<N$, Eq.~\eqref{eq:1} implies that 
\begin{align*}c_{\alpha, R} & = \frac{1}{d^N-1}\sum_{i=1}^N d^i\left(1-\sum_{j\ge 1}n_{\alpha_i, j}(R)d^{-i}\right)\\ &\ge  \frac{1}{d^N-1}\sum_{i=1}^{N-2} d^i\left(1-\frac{1}{d}\right)=\frac{d^{N-2}-1}{d^N-1 }\ge\frac{1}{d^2}-\frac{1}{d^3}.\end{align*}

It remains to consider the cases where $e=2$ and $N=1, 2$.  Eq.~\eqref{eq:1} and the bound $\sum_{j\ge 1}n_{\alpha_i, j}(R)\le d'/e+1\le d$ yield the following trivial configurations:
\begin{itemize}
\item $e=2, N=1$ and $n_{\alpha, 1}(R)=d$;
\item $e=2, N=2$ and $n_{\alpha_1, 1}(R)=n_{\alpha_2, 1}(R)=d$.
\end{itemize} 
In both cases, it follows that $c_{\alpha, R}=0$ and so $\alpha$ is $R$-critical. Suppose that $\alpha$ does not satisfy none of the cases described above. For $N=1$,  the inequality $n_{\alpha, 1}(R)<d$ implies that
$$c_{\alpha, R}= \frac{d\theta_1}{d-1}\ge \frac{d\left(1-(d-1)\cdot d^{-1}-1\cdot d^{-2}\right)}{d-1}=\frac{1}{d}\ge \frac{1}{d^2}-\frac{1}{d^3}.$$
For $N=2$, recall that we are under the condition $(n_{\alpha_1, 1}(R), n_{\alpha_2, 1}(R))\ne (d, d)$. In particular, from the argument employed in the case $N=1$, the inequality $\theta_i\ge \frac{d-1}{d^2}$ holds for at least one index $i\in \{1, 2\}$. Therefore, 
$$c_{\alpha, R}= \frac{d\theta_1+d^2\theta_2}{d^2-1}\ge \frac{d\cdot \tfrac{d-1}{d^2}+d^2\cdot 0}{d^2-1}=\frac{1}{d(d+1)}\ge \frac{1}{d^2}-\frac{1}{d^3}.$$

\item Suppose that $e=1$. Since $R^{-\infty}(\alpha)$ can contain an $R$-trivial element, Eq.~\eqref{eq:1} implies that \begin{equation}\label{eq:1'}\sum_{i=1}^{N}\sum_{j\ge 1}n_{\alpha_i, j}(R)\le 2d-1+N.\end{equation} We recall that $n_{\alpha_i, 1}(R)\ge 1$. For $N\ge 3$, it follows that
\begin{align*}c_{\alpha, R}& \ge \frac{1}{d^N-1}\sum_{i=1}^Nd^i\left(1-d^{-1}\sum_{j\ge 1}n_{\alpha_i, j}(R)\right) \\ &\ge \frac{1}{d^N-1}\left(\sum_{i=1}^{N-2}d^i\left(1-d^{-1}\right)-d^{N-2}d^{-1}\right)=\frac{d^{N-2}(d-1)-d}{d(d^N-1)}>\frac{1}{4d^2},\end{align*}
provided that $d>2$ if $N=3$. If $(d, N)=(2, 3)$, Eq.~\eqref{eq:1} yields the trivial configuration $n_{\alpha_1,1}(R)=n_{\alpha_{2},1}(R)=n_{\alpha_3,1}(R)=2$, in which $c_{\alpha, R}=0$ and so $\alpha$ is $R$-critical. If $(d, N)=(2, 3)$ and $\alpha$ is not $R$-critical, then $n_{\alpha_i, 1}(R)=1$ for at least one index $i\in\{1, 2, 3 \}$. In particular, $\theta_i\ge (1-2^{-1}-2^{-2})=\tfrac{1}{4}$ for at least one index $i\in \{1, 2, 3\}$ and so
$$c_{\alpha, R}= \frac{2\theta_1+4\theta_2+8\theta_3}{7}\ge\frac{2\cdot\tfrac{1}{4}+4\cdot0+8\cdot 0}{7}=\frac{1}{14}>\frac{1}{4\cdot 2^2}.$$
 For $N=2$, Eq.~\eqref{eq:1'} yields the following trivial configurations:
\begin{itemize}
	\item $e=1, n_{\alpha_1, 1}(R)=n_{\alpha_2, 1}(R)=d$;
	\item $e=1, d=2$, $n_{\alpha_2, 1}(R)=2$, $n_{\alpha_1, 1}(R)=1$ and $n_{\alpha_1, 2}(R)=2$;
	\item $e=1, d=2$, $n_{\alpha_2, 1}(R)=1$, $n_{\alpha_2, 2}(R)=2$ and $n_{\alpha_1, 1}(R)=2$.
\end{itemize} 
In these cases, it follows that $c_{\alpha, R}=0$ and so $\alpha$ is $R$-critical. Suppose that $\alpha$ does not satisfy none of the cases described above. For $d>2$, we employ the same argument used in the case $e=N=2$ and obtain that $$c_{\alpha, R}\ge \frac{d(1-(d-1)d^{-1}-2d^{-2})}{d^2-1}=\frac{d-2}{d(d^2-1)}>\frac{1}{4d^2}.$$ 
For $d=2$ we have that $\theta_i\ge (1-2^{-1}-2^{-2}-2^{-3})=\tfrac{1}{8}$ for at least one index $i\in \{1, 2\}$, hence 
$$c_{\alpha, R}= \frac{2\theta_1+4\theta_2}{3}\ge \frac{2\cdot\tfrac{1}{8}+4\cdot0}{3}=\frac{1}{12}>\frac{1}{4\cdot 2^2}.$$
For $N=1$,  Eq.~\eqref{eq:1'} yields the following trivial configurations:
\begin{itemize}
	\item $e=1, n_{\alpha, 1}(R)=d$;
	\item $e=1, n_{\alpha, 1}(R)=d-1$ and $n_{\alpha, 2}(R)=d$.
\end{itemize} 
In both cases, it follows that $c_{\alpha, R}=0$ and so $\alpha$ is $R$-critical. If $\alpha$ does not satisfy any of the cases described above, then either $n_{\alpha, 1}(R)<d$ or $n_{\alpha, 1}(R)=d$ and $n_{\alpha, 2}(R)<d$. In particular, Eq.~\eqref{eq:1'} implies that
$$c_{\alpha, R}=\frac{d\theta_1}{d-1}\ge \frac{d(1-d^{-1}(d-1)-d^{-2}(d-1)-2d^{-3})}{d-1}=\frac{d-2}{d^2(d-1)}>\frac{1}{4d^2},$$
whenever $d>2$. For $d=2$,  Eq.~\eqref{eq:1'} yields the trivial configuration $n_{\alpha,1}(R)=n_{\alpha,2}(R)=1$ and $n_{\alpha,3}(R)=2$, in which $c_{\alpha, R}=0$ and so $\alpha$ is $R$-critical. If $\alpha$ is not $R$-critical, then
$$c_{\alpha, R}=2\theta_1\ge 2(1-2^{-1}-2^{-2}-2^{-3}-2^{-4})=\frac{1}{8}> \frac{1}{4\cdot 2^2}.$$
\end{enumerate}
We combine all the previous bounds and obtain that if $\alpha$ is $R$-periodic and not $R$-critical, then $c_{\alpha, R}\ge \frac{1}{4d^2}$. This completes the proof of Theorem~\ref{thm:main}.

\subsection{Proof of Theorem~\ref{thm:sec}} 
Let $R=g/h$ be a rational function of degree $D$ whose $p$-reduction has degree $d\ge 1$. If $d=1$ it is direct to verify that $R(x)=\frac{ax^D+b}{cx^D+d}$ with $ad-bc\ne 0$ and either $D=1$ or $\K$ has characteristic $p>0$ and $D$ is a power of $p$.  Hence for every $n\ge 0$ we have that $R^{(n)}(x)=\frac{a_nx^{D^n}+b_n}{c_nx^{D^n}+d_n}$, where $a_n, b_n, c_n, d_n\in \K$ with $a_nd_n-b_nc_n\ne 0$. Since $\K$ is perfect it follows that for every $\alpha\in \ok$ and every $n\ge 0$, the equation $R^{(n)}(x)=\alpha$ has at most $1$ solution in $\ok$. Hence every $\alpha\in \ok$ is $R$-critical and $R^{-\infty}(\alpha)$ is finite if and only if $\alpha$ is $R$-periodic or $c\ne 0$ and $R^{-\infty}(\alpha)$ contains the $R$-trivial element $\beta=\frac{a}{c}$.

For $d>1$, Proposition~\ref{prop:est} entails that $\alpha$ is $R$-critical if and only if $c_{\alpha, R}=0$. From the proof of Theorem~\ref{thm:main}, we list the possible numerical configurations that yields $c_{\alpha, R}=0$. As follows, we present them in the order that they appear.

\begin{enumerate}[I.]
\item  $\alpha$ is not $R$-periodic and
\begin{enumerate}[(a)]
	\item $n_{\alpha, 1}(R)=d$;
	\item $n_{\alpha, 1}(R)=d-1$ and $n_{\alpha, 2}(R)=d$.
\end{enumerate}
\item  $\alpha=\alpha_N$ is $R$-periodic with period $N$ and
\begin{enumerate}[(a)]
	\item $e=2$, $N=1$ and $n_{\alpha_1, 1}(R)=d$;
	\item $e=2$, $N=2$ and $n_{\alpha_1, 1}(R)=n_{\alpha_2, 1}(R)=d$;
	\item $e=1$, $N=3$, $d=2$ and $n_{\alpha_1,1}=n_{\alpha_{2},1}=n_{\alpha_3,1}=2$;
	\item $e=1$, $N=2$ and $n_{\alpha_1, 1}(R)=n_{\alpha_2, 1}(R)=d$;
	\item $e=1$, $N=2$, $d=2$, $n_{\alpha_2, 1}(R)=2$, $n_{\alpha_1, 1}(R)=1$ and $n_{\alpha_1, 2}(R)=2$;
	\item $e=1$, $N=2$, $d=2$, $n_{\alpha_2, 1}(R)=1$, $n_{\alpha_2, 2}(R)=2$ and $n_{\alpha_1, 1}(R)=2$;
	\item $e=1$, $N=1$ and $n_{\alpha, 1}(R)=d$;
	\item $e=1$, $N=1$ and $n_{\alpha, 1}(R)=d-1$ and $n_{\alpha, 2}(R)=d$;
	\item $e=1$, $N=1$, $d=2$, $n_{\alpha,1}=n_{\alpha,2}=1$ and $n_{\alpha,3}=2.$
\end{enumerate}
\end{enumerate}

\begin{remark}\label{remark:1}
	If $\beta\in\ok$ is not an $R$-periodic element, then $n_{\beta, 1}(R)=d$ if and only if $\beta$ is the $R$-critical element and $\deg(g-\beta h)=0$. In this case, $\beta\in \K$ and there exists $\lambda\in\K^*$ such that $g(x)=\beta h(x)+\lambda$.
\end{remark}
We characterize the pairs $(\alpha, R)$ satisfying the numerical conditions above and explicitly exhibit the set $R^{-\infty}(\alpha)$ in the corresponding case. In order to simplify calculations, we frequently use the fact that $\frac{g}{h}=\frac{ag}{ah}$ for every $a\in \overline{\K}^*$.

\begin{enumerate}[I.]
\item \begin{enumerate}[(a)]
\item  Since $r_{\alpha,R}=d-n_{\alpha, 1}(R)=0$ and $\alpha$ is not $R$-periodic, it follows that $g(x)-\alpha h(x)=\lambda$ for some $\lambda\in\K^*$. Therefore, $R(x)=\alpha+\frac{\lambda}{h(x)}$ for some $h\in \K[x]$ of degree $D$. In this case, $R^{-\infty}(\alpha)=\{\alpha\}$.
\item Since $r_{\alpha,R}=d-n_{\alpha, 1}(R)=1$ and $n_{\alpha, 2}(R)=d$, Remark \ref{remark:1} entails that
$$\begin{cases}
	 g(x)-\alpha h(x)=(x-\beta)^D;\\
	 g(x)-\beta h(x)=\lambda,\\
\end{cases}$$
for some $\beta\in \K\setminus \{\alpha\}$ and some $\lambda\in \K^*$. By solving this system of equations, we obtain that $R(x)=\beta+\frac{\lambda}{(x-\beta)^D-\frac{\lambda}{\beta-\alpha}}$. In this case, $R^{-\infty}(\alpha)=\{\alpha, \beta\}$.
\end{enumerate}
\item \begin{enumerate}[(a)]
\item Since $e=2$, $\alpha$ is not an $R$-critical element. Since $r_{\alpha, R}=d+1-n_{\alpha, 1}=1$ and $N=1$, we obtain that $g(x)-\alpha h(x)=(x-\alpha)^D$. If $\tau $ is the unique non trivial $\K$-automorphism of $\K(\alpha)$, it follows that $g(x)-\overline{\alpha} h(x)=(x-\overline{\alpha})^D$ with $\overline{\alpha}=\tau(\alpha)$. We conclude that $R(x)=\frac{\overline{\alpha}(x-\alpha)^D-\alpha(x-\overline{\alpha})^D}{(x-\alpha)^D-(x-\overline{\alpha})^D}$ and $R^{-\infty}(\alpha)=\{\alpha\}$.  
\item Since $e=2$, $\alpha$ is not an $R$-critical element. Since $N=2$ and $r_{\alpha_i, R}=n_{\alpha_i, 1}-d+1=1$ for $i=1, 2$, we obtain that $g(x)-\alpha_1 h(x)=(x-\alpha_{2})^D$ and $g(x)-\alpha_2 h(x)=\lambda (x-\alpha_1)^D$ for some $\lambda\in \overline{\K}$. Arguing similarly to item II-(a), we necessarily have that $(\alpha_1, \alpha_2)=(\overline{\alpha}, \alpha)$ and $\lambda=1$. The latter implies that $R(x)=\frac{\alpha(x-\alpha)^D-\overline{\alpha}(x-\overline{\alpha})^D}{(x-\alpha)^D-(x-\overline{\alpha})^D}$ and so $R^{-\infty}(\alpha)=\{\alpha, \overline{\alpha}\}$.
\item Let $\tilde{R}=\tilde{g}/\tilde{h}$ be the $p$-reduction of $R$. We observe that $r_{\alpha_i,R}=2+1-n_{\alpha_i, 1}(R)=1$ for $i\in \{1, 2, 3\}$. Following the proof of item (i) in Proposition~\ref{lemma:1}, the latter entails that one of the elements $\alpha_i$ is $R$-critical and 
$$\begin{cases}
\tilde{g}(x)-\sigma_R(y_1) \tilde{h}(x)=(x-y_2)^2;\\
\tilde{g}(x)-\sigma_R(y_2) \tilde{h}(x)=\gamma(x-y_3)^2;\\
\tilde{g}(x)-\sigma_R(y_3) \tilde{h}(x)=\lambda(x-y_1),\\
\end{cases}$$
where $\gamma,\lambda\in\K^*$ and $\{y_1,y_2,y_3\}$ is a permutation of $\{\alpha_1,\alpha_2,\alpha_3\}$. These system above implies that $2y_1=y_2+y_3$, $\gamma=\frac{\sigma_R(y_2)-\sigma_R(y_3)}{\sigma_R(y_1)-\sigma_R(y_3)}$ and $\lambda=(\sigma_R(y_2)-\sigma_R(y_3))(2y_2-2y_3)$. 
If $\K$ has characteristic $2$ we have that $y_2=y_3$, a contradiction. Hence $\K$ does not have characteristic $2$ and so $y_1=\frac{y_2+y_3}{2}, \gamma=2$. We return to the initial equations, and after some calculations we obtain that 
$$\tilde{R}(x)=\frac{\sigma_R(y_2)(x-y_2)^2-\sigma_R(y_2+y_3)(x-y_3)^2}{(x-y_2)^2-2(x-y_3)^2}.$$
Since $R=\tilde{R}^{D/2}$ and either $D/2=1$ or $\K$ has characteristic $p>0$ and $D/2=p^h$, it follows by the definition of $\sigma_R$ that
$$R(x)=\frac{y_2(x-y_2)^D-(y_2+y_3)(x-y_3)^D}{(x-y_2)^D-2(x-y_3)^D}.$$
Moreover, $\alpha\in \{y_2, y_3, \frac{y_2+y_3}{2}\}=R^{-\infty}(\alpha)$.

\item Similarly to the case II-(b) we have that if $(\alpha_1, \alpha_2)=(\beta, \alpha)$, then $g(x)-\beta h(x)=(x-\alpha)^A$ and $g(x)-\alpha h(x)=\lambda (x-\beta)^B$ for some $\lambda\in \K^*$ and some integers $A, B\ge 1$ with $\max\{A, B\}=D$. The latter implies that $R(x)=\frac{\alpha(x-\alpha)^A-\beta \lambda (x-\beta)^B}{(x-\alpha)^A-\lambda(x-\beta)^B}$ and so $R^{-\infty}(\alpha)=\{\alpha, \beta\}$.

\item Set $\eta=\alpha_1$. Since $r_{\alpha,R}=d+1-n_{\alpha, 1}(R)=1$ and $r_{\eta, R}=d+1-n_{\alpha_2, 1}(R)=2$, there exists an element $\beta$ that is not $R$-periodic with $R(\beta)=\eta$. Moreover, we have that $r_{\beta,R}=0$ and then 
$$\begin{cases}
\tilde{g}(x)-\sigma_R(\alpha)\tilde{h}(x)=(x-\eta)^2;\\
\tilde{g}(x)-\sigma_R(\eta)\tilde{h} (x)=\gamma(x-\alpha)(x-\beta);\\
\tilde{g}(x)-\sigma_R(\beta) \tilde{h}(x)=\lambda,\\
\end{cases}$$
for some $\lambda,\gamma\in\K^*$, where $\tilde{R}=\tilde{g}/\tilde{h}$ is the $p$-reduction of $R$. These equations imply that $2\sigma_R(\eta)=\sigma_R(\alpha)+\sigma_R(\beta)$, $\gamma=\frac{\sigma_R(\beta)-\sigma_r(\eta)}{\sigma_R(\beta)-\sigma_R(\alpha)}$ and $\eta^2+\lambda\frac{\sigma_R(\eta)-\sigma_R(\alpha)}{\sigma_R(\beta)-\sigma_R(\eta)}=\alpha\beta$. If $\K$ has characteristic $2$ the latter entails that $\alpha=\beta$, a contradiction. Hence $\K$ does not have characteristic $2$ and so $
\eta=\frac{\alpha+\beta}{2}$, $\gamma=\frac{1}{2}$ and $\lambda=-\frac{(\alpha-\beta)^2}{4}$. We return to the initial equations, and after some calculations we obtain that 
$$\tilde{R}(x)=\sigma_R(\beta)+\frac{(\alpha-\beta)^2(\sigma_R(\alpha)-\sigma_R(\beta))}{(2x-\alpha-\beta)^2+(\alpha-\beta)^2}.$$
Since $R=\tilde{R}^{D/2}$ and either $D/2=1$ or $\K$ has characteristic $p>0$ and $D/2=p^h$,  it follows by the definition of $\sigma_R$ that
$$R(x)=\beta+\frac{(\alpha-\beta)^{D+1}}{(2x-\alpha-\beta)^D+(\alpha-\beta)^D}.$$
In this case, $R^{-\infty}(\alpha)=\{\alpha, \beta, \frac{\alpha+\beta}{2}\}$.

\item This case is entirely similar to item II-(e). We conclude that $\K$ does not have characteristic $2$ and 
$$R(x)=\beta+\frac{2(\alpha-\beta)^{D+1}}{(x-\alpha)^D+(\alpha-\beta)^D}.$$
Moreover, $R^{-\infty}(\alpha)=\{\alpha, \beta, 2\alpha-\beta\}$.

\item Since $r_{\alpha, R}=d+1-n_{\alpha, 1}(R)=1$ and $N=1$, it follows that $g(x)-\alpha h(x)=(x-\alpha)^A$ for some $1\le A\le D$. We conclude that $R(x)=\alpha+\frac{(x-\alpha)^A}{h(x)}$ for some $h\in \K[x]$ with $h(\alpha)\ne 0$ and $\max\{\deg(h), A\}=D$. Moreover, $R^{-\infty}(\alpha)=\{\alpha\}$.

\item Since $r_{\alpha,R}=d+1-n_{\alpha, 1}(R)=2$ and $n_{\alpha, 2}(R)=d$, there exists $\beta\in\K\setminus \{\alpha\}$ such that $R(\beta)=\alpha$ and $r_{\beta,R}=0$. Therefore, $g(x)-\alpha h(x)=(x-\beta)^A(x-\alpha)^{D-A}$ and 
$g(x)-\beta h(x)=\lambda$ for some $\lambda\in\K^*$ and some integer $1\le A<D$. The latter implies that $R(x)=\frac{\beta(x-\beta)^A(x-\alpha)^{D-A}-\alpha\lambda}{(x-\beta)^A(x-\alpha)^{D-A}-\lambda}$.
Moreover, $R^{-\infty}(\alpha)=\{\alpha, \beta\}$.

\item This case is entirely similar to item II-(e). We conclude that $\K$ does not have characteristic $2$ and
$$R(x)=\frac{\alpha+\beta}{2}+\frac{(\alpha-\beta)^{D+1}}{4(2x-\alpha-\beta)^D-2(\alpha-\beta)^D}.$$
Moreover, $R^{-\infty}(\alpha)=\{\alpha, \beta, \frac{\alpha+\beta}{2}\}$. 
\end{enumerate}
\end{enumerate}

In particular, if $\alpha$ is $R$-critical, then $\Delta_{\alpha, R}(n)\le 2$ for every $n\ge 0$. Moreover, for $d>1$, we have verified that the set $R^{-\infty}(\alpha)$ is finite. The proof of Theorem~\ref{thm:sec} is complete.

\section{Further results in the finite field setting}
Throughout this section, $\fq$ denotes the finite field of $q$ elements, where $q$ is a prime power. Let $\mathcal M_q$ be the set of monic polynomials $f\in \F_q[x]$ of positive degree, without any root in $\F_q$. 

\begin{definition}\label{def:arith}
Given a rational function $R=g/h\in \F_q(x)$ of degree $D\ge 1$ and $f\in \mathcal M_q$, we set $f_R=h^{\deg(f)}\cdot f\left(\frac{g}{h}\right)$. For each $n\ge 0$, the $n$-th $R$-transform of $f$ is the polynomial $f_R^{(n)}$ defined by $f_R^{(0)}=f$ and $f_R^{(n)}=(f_{R}^{(n-1)})_R$ if $n\ge 1$.
Moreover, let 
$$f_R^{(n)}(x)=p_{1, n}(x)^{e_{1, n}}\ldots p_{N_n, n}(x)^{e_{N_n, n}},$$
be the irreducible factorization of $f_R^{(n)}$ in $\F_q[x]$. We define the following arithmetic functions
\begin{enumerate}[(a)]
\item $\delta_{f, R}(n)=\deg(p_{1, n}(x)\cdots p_{N_n, n}(x))$ is the degree of the squarefree part of $f_R^{(n)}$;
\item $M_{f, R}(n)=\max\limits_{1\le i\le N_n} \deg(p_{i, n}(x))$ is the largest degree of an irreducible factor of $f_R^{(n)}$ over $\F_q$;
\item $N_{f, R}(n)=N_n$ is the number of distinct irreducible factors of $f_R^{(n)}$ over $\F_q$;
\item $A_{f, R}(n)=\frac{\Delta_{f, R}(n)}{N_{f, R}(n)}$ is the average degree of the distinct irreducible factors of $f_R^{(n)}$ over $\F_q$.
\end{enumerate}
\end{definition}
 
The above naturally extends Definition~1.2 in~\cite{R}, where $R=g$ is a polynomial. In~\cite{R} the author explores the growth (linear, polynomial, exponential) of the functions above, among some others. Our aim here is to discuss the growth of these arithmetic functions in the context of rational functions. For functions $\mathcal F, \mathcal G:\mathbb N\to \mathbb R_{>0}$, we write $\mathcal F\gg \mathcal G$ if there exists $c>0$ such that $c\cdot \mathcal F(n)\ge\mathcal G(n)$ for every $n$ sufficiently large. We also write $\mathcal F\approx \mathcal G$ if $\mathcal F\gg\mathcal G$ and $\mathcal G\gg \mathcal F$. We have the following result.
\begin{lemma}\label{lem:rt}

Given a rational function $R=g/h\in \F_q(x)$ of degree $D\ge 1$. Then for every $f\in \mathcal M_q$, the polynomials $f_R^{(n)}$ and $f_{R^{(n)}}$ have the same roots. In particular, if $\alpha_1, \ldots, \alpha_s\in \overline{\F}_q$ are the distinct roots of $f$, we have that 
$$\delta_{f, R}(n)=\sum_{i=1}^s\Delta_{\alpha_i, R}(n).$$
\end{lemma}

\begin{proof} 
It suffices to prove the first statement. We proceed by induction on $n$. The cases $n=0, 1$ follow directly by the definition. Suppose that the result holds for some $n\ge 1$ and let $N=n+1$. We observe that, for every $k\ge 0$ and every $F\in \mathcal M_q$, the roots of $F_{R^{(k)}}$ comprise the solutions of the equations $R^{(k)}(x)=\alpha$ with $\alpha$ running over the roots of $F$. In particular, if $\beta\in\overline{\F}_q$ is a root of $f_{R^{(N)}}$, then $R(\beta)$ is a root of $f_{R^{(n)}}$. From induction hypothesis, $R(\beta)$ is a root of $f_R^{(n)}$, hence $\beta$ is a root of $(f_R^{(n)})_{R}=f_{R}^{(N)}$. This proves that every root of $f_{R^{(N)}}$ is also a root of $f_R^{(n)}$. The converse follows in a similar way, proving the result.
\end{proof}

Combining Lemma~\ref{lem:rt} with Theorems~\ref{thm:main} and~\ref{thm:sec}, we obtain the following result.

\begin{corollary}\label{cor:growth}
Let $R\in \F_q(x)$ be a rational function whose $p$-reduction has degree $d>1$. If $f\in\mathcal M_q$ has at least one root $\alpha$ that is not $R$-critical, then there exists a constant $0<c_{f, R}\le \deg(f)$ such that $$\delta_{f, R}(n)=c_{f, R}\cdot d^n+O_{f, R}(1).$$ In this case, $M_{f, R}(n)\gg n$. In particular, any $f\in \mathcal M_q$ having at least one root in the set $\overline{\F}_q\setminus \F_{q^2}$ satisfies the above.
\end{corollary}

\begin{proof}
Pick $n$ large so that $\delta_{f, R}>0$. Let $m_n=M_{f, R}(n)$, hence the roots of $f_n^R$ all lie in the set $\bigcup_{1\le j\le m_n}\F_{q^j}$. Therefore, 
$$\delta_{f, R}(n)\le \sum_{j=1}^{m_n}q^j<q^{m_n+1},$$
and so $m_n\ge \frac{\log \delta_{f, R}(n)}{\log q}-1\gg n$ since $\delta_{f, R}\gg d^n$ and $d>1$. Moreover, from Theorem~\ref{thm:sec}, we have that any $R$-critical element lies in $\F_{q^2}$ if the $p$-reduction of $R$ has degree $d>1$.
\end{proof}

Corollary~\ref{cor:growth} entails that under mild conditions on $(f, R)$, the arithmetic function $M_{f, R}(n)$ grows at least linearly with respect to $n$. When $R=g$ is a polynomial, we recover Lemma~4.4 in~\cite{R}. According to~\cite{R}, this lower bound is optimal on the growth type. More precisely, if $f\in \F_q[x]$ has positive degree, for infinitely many polynomials $g$ we have that $M_{f, g}(n), A_{f, g}(n)\approx n$. The family of polynomials $g$ taken there comprise linearized polynomials $\sum_{i=0}^ta_ix^{q^i}$. For more details, see Proposition~5.18 in~\cite{R}. As follows, we prove that this bound is also optimal for rational functions that are not polynomials. Our main idea is to conjugate a polynomial with a Mobius map in a way that the resulting rational function is not a polynomial. We need the following technical lemmas.

\begin{lemma}[\cite{ST12}]\label{action}
For $[A]\in \PGL(2, q)$ with $A=\left(\begin{matrix}
a&b\\
c&d
\end{matrix}\right)$, $\alpha\in \overline{\F}_{q}\setminus \F_q$ and $f\in \mathcal \F_q[x]$ of degree $k\ge 1$, set $[A]\circ f(x)=(bx+d)^{k}f\left(\frac{ax+c}{bx+d}\right)$ and $[A] \ast \alpha = \frac{d\alpha-c}{-b \alpha +a}$. Then for $f\in \mathcal M_q$, the polynomial $[A]\circ f$ has degree $k$ and, if $\alpha \in \overline{\F}_q\setminus \F_q$, we have that $f(\alpha) = 0 \iff ([A] \circ f)([A]\ast \alpha)=0$. 
\end{lemma}

\begin{lemma}\label{lem:same}
For $[A]\in \PGL(2, q)$ with $A=\left(\begin{matrix}
a&b\\
c&d
\end{matrix}\right)$ and $R\in \overline{\F}_{q}(x)\setminus \F_q$, set $[A]\bullet R=\frac{aR+c}{bR+d}$. This defines an action of $\PGL(2, q)$ on the set $\overline{\F}_{q}(x)\setminus \F_q$.  If $g\in \F_q[x]$ has degree $k\ge 1$ and $g^A(x):=[A]\bullet (g([A^{-1}]\bullet x))\in \F_q(x)$, then for every $n\ge 1$ and every $f\in \mathcal M_q$, we have that $$\mathcal F_{f, g^A}(n)=\mathcal F_{[A]\circ f, g}(n),$$ where $\mathcal F$ is any of the four arithmetic functions in~Definition~\ref{def:arith}.
Moreover, for $b = 0$, $g^A\in \F_q(x)\setminus \F_q[x]$ is a rational function of degree $k$.

\end{lemma}

\begin{proof}
It is direct to verify that, for every $[A], [B]\in \PGL(2, q)$ and every $g\in \overline{\F}_q[x]\setminus \overline{\F}_q$, we have that $[A]\bullet g\in \overline{\F}_q(x)\setminus \F_q$ and $[A]\bullet ([B]\bullet g)=[AB]\bullet g$. In particular, $\PGL(2, q)$ acts on $\overline{\F}_{q}(x)\setminus \F_q$ via the compositions $[A]\bullet g$.
Pick $f\in \mathcal M_q$, let $n\ge 0$ be an integer and let $\Gamma_1, \Gamma_2$ be the set of distinct roots of $f_{g^A}^{(n)}$ and $([A]\circ f)_g^{(n)}$, respectively. Since $f\in \mathcal M_q$, we have that $\Gamma_1\cap \F_q=\emptyset$. We observe that the $n$-fold composition $(g^A)^{(n)}$ equals $(g^{(n)})^A$. Moreover, $[A]\ast \alpha=[A]^{-1}\bullet \alpha$ for every $\alpha\in \overline{\F}_q\setminus \F_q$. In particular, Lemmas~\ref{lem:rt} and~\ref{action} imply that $$\Gamma_1=\{[A]\bullet \beta\,|\, \beta \in \Gamma_2\}\subseteq  \overline{\F}_q\setminus \F_q.$$
Lemma~\ref{action} entails that the minimal polynomials of $\gamma$ and $[A]\bullet\gamma $ over $\F_q$ have the same degree for every $\gamma\in \overline{\F}_q\setminus \F_q$. Moreover, the map $y\mapsto y^q$ commutes with the map $y\mapsto [A]\bullet y$. From these observations, we conclude that $\mathcal F_{f, g}(n)=\mathcal F_{[A]\circ f, g^A}(n)$, where $\mathcal F$ is any of the four arithmetic functions defined in~Definition~\ref{def:arith}. 

It follows by the definition that $g^A(x)=\frac{a g([A]^{-1}\bullet x)+c}{bg([A]^{-1}\bullet x)+d}=\frac{a\cdot  [A^{-1}]\circ g(x)+c(a-bx)^k}{b\cdot [A^{-1}]\circ g(x)+d(a-bx)^k}$. In particular, if $b=0$, the rational function $g^A\in \F_q(x)$ is not a polynomial and has degree $k$.
\end{proof}

Combining Lemma~\ref{lem:same} with Theorem~2.6 of~\cite{R}, we obtain the following result.

\begin{theorem}\label{thm:gr}
For each $f\in \mathcal M_q$, the following hold:

\begin{enumerate}[(i)]
\item there exist infinitely many rational functions $R\in \F_q(x)\setminus \F_q[x]$ such that $M_{f, R}(n)\approx n$;
\item for each integer $t\ge 0$, there exist infinitely many rational functions $R\in \F_q(x)\setminus \F_q[x]$ such that $N_{f, R}(n)\approx n^t$ and $M_{f, R}(n)\approx \deg(R)^n$.
\end{enumerate}
\end{theorem}
From Proposition~5.18 of~\cite{R}, we can also extend item (i) of the previous theorem to the function $A_{f, R}(n)$. 
\subsection{Some open problems}
We end this section by extending some open problems that are proposed in~\cite{R}. In what follows, $R\in \F_q(x)$ is a rational function and $f\in \F_q[x]$ is a polynomial of positive degree with at least one root that is not $R$-critical. Theorem~\ref{thm:gr} implies that $M_{f, R}$ may have linear or exponential growth if $f\in \mathcal M_q$. We believe that these are the only possible cases.
\begin{problem}\label{p1}
Prove or disprove: either $M_{f, R}(n)\approx n$ or $\log M_{f, R}(n)\gg n$.
\end{problem}

We observe that $M_{f, R}(n)\cdot N_{f, R}(n)\ge \delta_{f, R}(n)$ for every $n\ge 0$. In particular, there exists $d_0>1$ such that for every $n\gg 1$, either $M_{f, R}(n)>d_0^n$ or $N_{f, R}(n)>d_0^n$. However, this is not sufficient to conclude that at least one of these functions have exponential growth. Motivated by these observations, we propose the following problem.

\begin{problem}\label{p2}
Prove or disprove: either $\log M_{f, R}(n)\gg n$ or $\log N_{f, R}(n)\gg n$.
\end{problem}

Since $M_{f, R}(n)\cdot N_{f, R}(n)\ge \delta_{f, R}(n)$, a positive answer to Problem~\ref{p1} implies a positive answer to Problem~\ref{p2}. 

\begin{problem}\label{p3}
Prove or disprove: $A_{f, R}(n)\gg n$.
\end{problem}
We have seen that $A_{f, R}(n)\approx n$ for infinitely many rational functions $R$. In particular, Positive answer to Problem~\ref{p3} implies that the bound $A_{f, R}(n)\gg n$ is sharp on the growth type.

\end{document}